\documentclass[11pt,reqno]{amsart}
\usepackage{amsmath,amsthm,amssymb,bm}
\usepackage[hypertex]{hyperref}
\usepackage[all]{xy}
\usepackage{tikz}
\usetikzlibrary{cd}

\makeatletter

\@addtoreset{equation}{section}
\makeatother

\newtheorem{Thm}{Theorem}[section]
\newtheorem{Lem}[Thm]{Lemma}

\newtheorem{Prop}[Thm]{Proposition}

\newtheorem*{Thm*}{Theorem}

\newcommand{\C}{\mathbb{C}}           
\newcommand{\Z}{\mathbb{Z}}
\newcommand{\Q}{\mathbb{Q}}

\newcommand{\Hom}{\mathrm{Hom} \,}

\newcommand{\id}{\mathrm{id}}

\newcommand{\cl}{\mathrm{cl}}

\newcommand{\fg}{{\mathfrak g}}

\newcommand{\fS}{{\mathfrak S}}

\newcommand{\ga}{\alpha}

\newcommand{\gl}{\lambda}
\newcommand{\gL}{\Lambda}
\newcommand{\gd}{\delta}
\newcommand{\gD}{\Delta}

\newcommand{\gs}{\sigma}

\newcommand{\ol}{\overline}

\newcommand{\wt}{\mathrm{wt}}

\newcommand{\fboxr}[1]{{\fbox{\raisebox{0pt}[7.5pt][0pt]{$#1$}}}}
\newcommand{\condC}[2]{\text{(G#1)}_{#2}}
\newcommand{\condD}[2]{\text{(D#1)}_{#2}}

\title[Existence of KR crystals of type $G_2^{(1)}$ and $D_4^{(3)}$]
{Existence of Kirillov-Reshetikhin crystals\\ of type $G_2^{(1)}$ and $D_4^{(3)}$}
\author{Katsuyuki Naoi}

\begin{document}

\address{%
Institute of Engineering \\
Tokyo University of Agriculture and Technology\\
2-24-16 Naka-cho, Koganei-shi, Tokyo 184-8588, JAPAN}
\email{naoik@cc.tuat.ac.jp}


\begin{abstract}
 In this paper we prove that every Kirillov-Reshetikhin module of type $G_2^{(1)}$ and $D_4^{(3)}$ has a crystal pseudobase 
 (crystal base modulo signs), by applying the criterion for the existence of a crystal pseudobase due to Kang et al.
\end{abstract}

\maketitle

\section{Introduction}

Let $\fg$ be an affine Kac-Moody Lie algebra, and $U_q'(\fg)$ the quantum affine algebra without the degree operator.
Among simple finite-dimensional $U_q'(\fg)$-modules there is a distinguished family called Kirillov-Reshetikhin (KR for short)
modules, which are parametrized by two integers $r,\ell$ and denoted by $W^{r,\ell}$.
Here $r$ corresponds to a node of the Dynkin diagram of $\fg$ except the prescribed $0$,
and $\ell$ is a positive integer.
KR modules are known to have nice properties, such as $T$ ($Q,Y$)-systems, fermionic character formulas, and so on 
(see \cite{MR1745263,MR1903978,MR1993360,MR2254805,MR2428305,MR2576287} and references therein).

As another nice property, it was conjectured in \cite{MR1745263,MR1903978} that a KR module has a crystal base in the sense of Kashiwara 
\cite{MR1115118}.
This conjecture has been established by Okado and Schilling \cite{MR2403558} when $\fg$ is of nonexceptional type, namely,
when the subalgebra $\fg_0$, whose Dynkin diagram is obtained from that of $\fg$ by removing $0$, is of classical type.
More precisely, they have proved that every KR module of nonexceptional type has a crystal pseudobase.
Here $B$ is said to be a crystal pseudobase if $B/\{\pm 1\}$ is a crystal base (for the precise definition,
see Subsection \ref{Subsection:crystal base and pseudobase} of the present paper).
On the other hand, when $\fg$ is of exceptional type the existence of a crystal pseudobase has not been proved in most cases: 
it is known by general theory \cite{MR1194953,MR1890649} that $W^{r,\ell}$ has a crystal (pseudo)base if one of the following 
conditions holds;
$W^{r,\ell}$ is simple as a $U_q(\fg_0)$-module, $r$ is a special node called the adjoint node, or $\ell = 1$ 
(the crystal structures in these cases are studied in \cite{zbMATH01251404,MR2199630,zbMATH05154980,zbMATH05226997,zbMATH06530852}, and so on).
As far as the author knows, in exceptional types the existence of a crystal pseudobase has been proved for these $W^{r,\ell}$ only.
A main difference between nonexceptional and exceptional types is that every KR module is multiplicity-free as a 
$U_q(\fg_0)$-module in the former types, but this is not true in the latter.

In this paper, we prove that every KR module of type $G_2^{(1)}$ and $D_4^{(3)}$ has a crystal pseudobase.
In both cases $\fg_0$ is of type $G_2$, and this gives a first proof of the existence of a crystal pseudobase of $W^{r,\ell}$ 
for general $r,\ell$ in an exceptional type.
Since one node of $\fg_0$ is adjoint in both cases, the KR modules associated with this node are already known to have crystal pseudobases.
Hence we focus on the ones associated with the other node, and prove the following.

\begin{Thm*}
 Assume that $\fg$ is of type $G_2^{(1)}$ or $D_4^{(3)}$, whose Dynkin diagram is labeled as in {\normalfont(\ref{eq:Dynkin})}.
 Then for any positive integer $\ell$, the KR module $W^{2,\ell}$ has a crystal pseudobase.
\end{Thm*}

Following \cite{MR2403558}, we prove this theorem by applying the criterion for the existence of a crystal pseudobase 
introduced in \cite{MR1194953}.
A KR module has a prepolarization (bilinear form having some properties) coming from the fusion construction,
and the criterion reduces the existence of a crystal pseudobase
to a statement concerning the values of the prepolarization of certain vectors (see Subsection \ref{Subsection:Reduction}).
However it seems hard to calculate the explicit values, 
and therefore we use the following method.
We consider $W^{2,\ell}$ as a submodule of a twist of $W^{2,\ell-1} \otimes W^{2,1}$, which enables us to express 
the values of the prepolarization on $W^{2,\ell}$ using that on $W^{2,\ell-1}$ and $W^{2,1}$.
Using this expression, we show the statement by the induction on $\ell$.

The paper is organized as follows. 
In Subsections \ref{Subsection:QAA}--\ref{Subsection:Polarization} we recall the basic notions.
In Subsection \ref{Subsection:Criterion} we recall the criterion for the existence of a crystal pseudobase,
and in Subsection \ref{Subsection:fusion} we review the fusion construction of a KR module.
We state the main theorem in Subsection \ref{Subsection:Main_Thm}, and then reduce it to a statement concerning the values
of the prepolarization of certain vectors in Subsection \ref{Subsection:Reduction}.
After collecting equalities used in the proof (Subsection \ref{Subsection:equalities}),
we show the statement for $G_2^{(1)}$ and $D_4^{(3)}$ separately
in Subsections \ref{Subsection:G_2^{(1)}} and \ref{Subsection:D_4^{(3)}}.

\section*{Acknowledgments}
The author is grateful to Masato Okado for answering his question concerning previous works, and to David Hernandez for helpful comments.
He is supported by JSPS Grant-in-Aid for Young Scientists (B) No.\ 16K17563.

\section{Preliminaries}\label{Section:pre}

\subsection{Quantum affine algebra}\label{Subsection:QAA}

Let $\fg$ be an affine Kac-Moody Lie algebra with index set $I=\{0,1,\ldots,n\}$ and Cartan matrix $C=(c_{ij})_{0\leq i,j \leq n}$.
Here the indices are ordered as in \cite[Section 4.8]{MR1104219}.
For simplicity, \textit{we assume that $\fg$ is not of type $A_{2n}^{(2)}$} (later we further assume that $\fg$ is of type 
$G_2^{(1)}$ or $D_4^{(3)}$).
Denote by $\ga_i$ and $h_i$ ($i\in I$) the simple roots and simple coroots respectively.
Let $\gL_i$ ($i \in I$) be the fundamental weights, $\gd$ the generator of null roots,
and $P = \bigoplus_i \Z \gL_i \oplus \Z\gd$ the weight lattice.
Let $W$ be the Weyl group, 
and $(\ , \ )$ a nondegenerate $W$-invariant bilinear form on $P$ satisfying
\[ \min\{ (\ga_i,\ga_i) \mid i \in I\} =2.
\]
Set 
\[ \varpi_i = \gL_i - \langle K, \gL_i\rangle \gL_0 \ \ \ \text{for $i \in I$},
\]
where $K$ denotes the canonical central element of $\fg$.
Set 
\[ P_{\cl} = P/\Z \gd,
\]
and let $\cl\colon P \twoheadrightarrow P_\cl$ be the canonical projection.
By abuse of notation, we also write $\ga_i$, $\gL_i$, $\varpi_i$ for $\cl(\ga_i)$, $\cl(\gL_i)$, $\cl(\varpi_i)$.

Let $q$ be an indeterminate, and set 
\[ [m]_q = \frac{q^m-q^{-m}}{q-q^{-1}}, \ \ \ [n]_{q}! = [n]_{q}[n-1]_{q}\cdots[1]_q, \ \ \ \begin{bmatrix} m \\ n \end{bmatrix}_{q}
= \frac{[m]_q[m-1]_q\cdots[m-n+1]_q}{[n]_q!}
\]
for $m \in \Z$, $n \in \Z_{\ge 0}$.
Set $q_i = q^{(\ga_i,\ga_i)/2}$ for $i\in I$.
The quantum affine algebra (without the degree operator) $U_q'(\fg)$ is the associative $\Q(q)$-algebra generated by $e_i$, $f_i$ ($i \in I$),
$q^{h}$ ($h \in P_\cl^* = \Hom(P_\cl,\Z)$), with the following relations;
\begin{align*}
 q^0=1&, \ \ \ q^{h_1}q^{h_2} = q^{h_1+h_2} \ \ \ \text{for $h_1,h_2 \in P_\cl^*$}, \\
 q^{h}e_iq^{-h}=q^{\langle h, \ga_i\rangle}e_i&, \ \ \ q^{h}f_iq^{-h}=q^{-\langle h,\ga_i\rangle}f_i \ \ \ \text{for $h \in P_\cl^*$, 
  $i \in I$},\\
 e_if_j-f_je_i = \gd_{ij}&\frac{t_i - t^{-1}_i}{q_i-q^{-1}_i} \ \ \ \text{for $i,j\in I$} \text{ where } t_i=q^{(\ga_i,\ga_i)h_i/2},
\end{align*}
and the Serre relations
\[ \sum_{k=0}^{1-c_{ij}} (-1)^ke_i^{(k)}e_je_i^{(1-c_{ij}-k)} = 0, \ \ \ 
   \sum_{k=0}^{1-c_{ij}} (-1)^kf_i^{(k)}f_jf_i^{(1-c_{ij}-k)} = 0 \ \ \ \text{for 
   $i,j \in I$ ($i\neq j$)}.
\]
Here we set $e_i^{(k)}=e_i^k/[k]_{q_i}!$, $f_i^{(k)}=f_i^k/[k]_{q_i}!$.

Denote by $\gD$ the coproduct of $U_q'(\fg)$ defined by
\begin{equation}\label{eq:coproduct}
 \gD(q^h) = q^h \otimes q^h,\ \ \ 
 \gD(e_i) = e_i \otimes t_i^{-1} + 1\otimes e_i, \ \ 
 \gD(f_i)= f_i\otimes 1 + t_i\otimes f_i
\end{equation}
for $h \in P_\cl^*$, $i\in I$.

For a $U_q'(\fg)$-module $M$ and $\gl \in P_\cl$, we write
\[ M_\gl = \{ v \in M \mid q^h{v} = q^{\langle h, \gl\rangle }v \ \text{ for $h \in P_\cl^*$}\}.
\]
If $v \in M_\gl$ with $v\neq 0$, we write $\wt(v) = \gl$.

\subsection{Crystal bases and pseudobases}\label{Subsection:crystal base and pseudobase}

We briefly recall the definition of crystal bases
(see \cite{MR1881971} for more details), and crystal pseudobases.
Let $A$ be the subring of $\Q(q)$ consisting of rational functions without poles at $q=0$.
Let $M$ be an integrable $U_q'(\fg)$-module, and $\tilde{e}_i,\tilde{f}_i$ ($i\in I$)
the Kashiwara operators.
A free $A$-submodule $L$ of $M$ is called a \textit{crystal lattice} of $M$ if\\[5pt]
(a) $M \cong \Q(q)\otimes_A L$, \ \ (b) $L = \bigoplus_{\gl \in P_\cl} L_\gl$ 
where $L_\gl = L \cap M_\gl$,\\
(c) $\tilde{e}_iL \subseteq L$, $\tilde{f}_iL \subseteq L$ ($i \in I$).\\[5pt]
A pair $(L,B)$ is called a \textit{crystal base} of $M$ if\\[5pt]
(i) $L$ is a crystal lattice of $M$, \ \ (ii) $B$ is a $\Q$-basis of $L/qL$,\\
(iii) $B=\bigsqcup_{\gl \in P_\cl} B_\gl$ where $B_\gl = B\cap \big(L_\gl/qL_\gl)$, \ \ (iv) $\tilde{e}_iB \subseteq B \cup \{0\}$,
$\tilde{f}_iB \subseteq B \cup \{0\}$,\\ (v) for $b,b'\in B$ and $i \in I$, $\tilde{f}_ib = b'$ if and only if $\tilde{e}_ib' = b$.\\[5pt]
We call $(L,B)$ a \textit{crystal pseudobase} of $M$ if they satisfy the conditions (i), (iii)--(v), and
 (ii') $B= B' \sqcup (-B')$ with $B'$ a $\Q$-basis of $L/qL$.

\subsection{Polarization}\label{Subsection:Polarization}

Let $M$ and $N$ be $U_q'(\fg)$-modules.
A bilinear pairing $( \ , \ ) \colon M \times N \to \Q(q)$ is said to be \textit{admissible} if it satisfies
\begin{equation}\label{eq:admissible}
 (q^hu,v) = (u,q^hv), \ \ \ (e_iu,v)=(u,q_i^{-1}t_i^{-1}f_iv), \ \ \ (f_iu,v)=(u,q_i^{-1}t_ie_iv)
\end{equation}
for $h \in P_\cl^*$, $i \in I$, $u \in M$, $v \in N$.
A bilinear form $(\ , \ )$ on $M$ is called a \textit{prepolarization} if it is symmetric and satisfies (\ref{eq:admissible}) for $u,v \in M$.
A prepolarization is called a \textit{polarization} if it is positive definite with respect to the following total order on $\Q(q)$:
\[ f > g \text{ if and only if } f-g \in \bigsqcup_{n\in \Z} \left\{ q^n(c+ qA)\mid c \in \Q_{>0}\right\},
\]
and $f \geq g $ if $f=g$ or $f>g$.
Throughout the paper, we use the notation $||u||^2 = (u,u)$ for $u \in M$.

\subsection{Criterion for the existence of a crystal pseudobase}\label{Subsection:Criterion}

In this subsection, we recall a criterion for the existence of a crystal pseudobase given by \cite{MR1194953}.

Let $\fg_0$ denote the simple Lie algebra whose Dynkin diagram is obtained from that of $\fg$ by removing the node $0$.
Let $I_0 = I\setminus \{0\}$, $P_0$ be the weight lattice of $\fg_0$, and $P_0^+\subseteq P_0$ the subset of dominant integral weights.
We identify $P_0$ with a subgroup of $P_\cl$ as follows:
\[ P_0 = \sum_{i\in I_0} \Z \varpi_i  \subseteq P_\cl.
\]
Let $U_q(\fg_0)$ be the $\Q(q)$-subalgebra of $U_q'(\fg)$ generated by $e_i$, $f_i$, $q^{h_i}$ ($i \in I_0$),
which is the quantized enveloping algebra of $\fg_0$.
For $\gl \in P_0^+$, let $V_0(\gl)$ denote the simple integrable $U_q(\fg_0)$-module with highest weight $\gl$.

Let $A_\Z$ and $K_\Z$ be the subalgebras of $\Q(q)$ defined respectively by
\[ A_\Z = \{f(q)/g(q) \mid f(q),g(q) \in \Z[q], g(0)=1\}, \ \ \ K_\Z=A_\Z[q^{-1}].
\]
Let $U_q'(\fg)_{K_\Z}$ denote the $K_\Z$-subalgebra of $U_q'(\fg)$ generated by $e_i,f_i,q^h$ ($i \in I, h \in P_\cl^*$).

Now the following proposition, which immediately follows from \cite[Propositions 2.6.1 and 2.6.2]{MR1194953}, is a key tool to show
our main theorem.

\begin{Prop}\label{Prop:criterion}
 Let $M$ be a finite-dimensional integrable $U_q'(\fg)$-module,
 and assume that $M$ has a prepolarization $(\ , \ )$ and a $U_q'(\fg)_{K_\Z}$-submodule $M_{K_\Z}$ such that $(M_{K_\Z},M_{K_\Z})\subseteq
 K_\Z$.
 Let $\gl_1,\ldots,\gl_m \in P_0^+$, and assume further that the following conditions hold:
 \begin{enumerate}
  \setlength{\parskip}{1pt} 
  \setlength{\itemsep}{1pt} 
  \item[(i)] $M \cong \bigoplus_{l=1}^m V_0(\gl_l)$ as $U_q(\fg_0)$-modules,
  \item[(ii)] there exist $u_k \in (M_{K_\Z})_{\gl_k}$ $(1\leq k \leq m)$ such that $(u_k,u_l) \in \gd_{kl} +qA$,
  and $||e_iu_k||^2\in q_i^{-2\langle h_i,\gl_k\rangle-2}qA$ for all $i \in I_0$.
 \end{enumerate}
 Then $(\ , \ )$ is a polarization.
 Moreover, if we set
 \[ L = \{u \in M \mid ||u||^2 \in A\} \text{ and } B= \{b \in (M_{K_\Z} \cap  L)/(M_{K_\Z} \cap qL)\mid (b,b)_0 = 1\}
 \]
 where $(\ , \ )_0$ is the $\Q$-valued bilinear form on $L/q L$ induced by $(\ , \ )$, then $(L,B)$ is a crystal pseudobase of $M$.
\end{Prop}

\subsection{Fundamental representations and fusion construction of KR modules}\label{Subsection:fusion}

For $r \in I_0$, denote by $W(\varpi_r)$ the \textit{fundamental representation} introduced by \cite{MR1890649}.
$W(\varpi_r)$ are known to have the following properties.

\begin{Prop}[{\cite[Propositions 2.4 and 2.6]{MR2403558}}]\label{Prop:fund.rep.} \
 \begin{enumerate}
  \setlength{\parskip}{1pt} 
  \setlength{\itemsep}{1pt} 
 \item[(i)] $W(\varpi_r)$ is a finite-dimensional simple integrable $U_q'(\fg)$-module.
 \item[(ii)] $\dim W(\varpi_r)_{\varpi_r} = 1$.
 \item[(iii)] The weight set of $W(\varpi_r)$, which is a subset of $P_\cl$,
 is contained in the intersection of $\varpi_r + \sum_{i \in I} \Z \ga_i$ 
 and the convex hull of $W\varpi_r$.
 \item[(iv)] $W(\varpi_r)$ has a crystal base.
 \item[(v)] $W(\varpi_r)$ has a polarization $(\ , \ )$.
 \end{enumerate}
\end{Prop}

For a $U_q'(\fg)$-module $M$, denote by $M_{\mathrm{aff}}$ the $U_q'(\fg)$-module $\Q(q)[z,z^{-1}]\otimes M$
on which $e_i$ and $f_i$ 
act by $z^{\gd_{0i}}\otimes e_i$ and $z^{-\gd_{0i}}\otimes f_i$, respectively.
For a nonzero $a \in \Q(q)$, set $M_a = M_{\mathrm{aff}}/(z-a)M_{\mathrm{aff}}$.
For $v \in M$, denote by $v_a \in M_a$ the image of $1 \otimes v$ under the projection $M_{\mathrm{aff}} \twoheadrightarrow M_a$. 
We simply write $v$ for $v_a$ when no confusion is possible. 
The following lemma is a consequence of \cite[Proposition 9.3]{MR1890649} (note that it is also proved in the paper that $W(\varpi_r)$ is a
 ``good" module).

\begin{Lem}
 Assume that $a,b \in \Q(q)$ are nonzero and satisfy $a/b \in A$. 
 Then for any $r \in I_0$, there exists a unique nonzero $U_q'(\fg)$-module homomorphism {\normalfont(}the normalized $R$-matrix{\normalfont)}
 \[ R(a,b) \colon W(\varpi_r)_a \otimes W(\varpi_r)_b \to W(\varpi_r)_b \otimes W(\varpi_r)_a
 \]
 satisfying $R(a,b)\big(u_a \otimes u_b\big) = u_b\otimes u_a$, 
 where $u$ is a nonzero vector of $W(\varpi_r)$ with weight $\varpi_r$.
\end{Lem}

Now let us recall the fusion construction of Kirillov-Reshetikhin modules.
Fix $r \in I_0$ and $\ell \in \Z_{>0}$, and let $w \in W(\varpi_r)$ be a vector with weight $\varpi_r$ and satisfying $||w||^2 =1$.
Let $\mathfrak{S}_\ell$ be the $\ell$-th symmetric group, and $s_i \in \fS_\ell$ the adjacent transposition interchanging $i$ and $i+1$.
For $\gs \in \mathfrak{S}_\ell$, denote by $\ell(\gs) \in \Z_{\ge 0}$ the length of $\gs$.
Let $x_1,\ldots,x_\ell \in \Q(q)\setminus\{0\}$ be such that $x_i/x_j \in A$ for all $i<j$.
Then for any $\gs \in \fS_\ell$, we can construct a well-defined $U_q'(\fg)$-module homomorphism 
\[ R_\gs(x_1,\ldots,x_\ell) \colon W(\varpi_r)_{x_1} \otimes \cdots \otimes W(\varpi_r)_{x_\ell}
   \to W(\varpi_r)_{x_{\gs(1)}} \otimes \cdots \otimes W(\varpi_r)_{x_{\gs(\ell)}}
\]
by
\begin{align*}
 R_1(x_1,\ldots,x_\ell) =& \id\\
 R_{s_i}(x_1,\ldots,x_\ell) =& \left(\bigotimes_{j<i} \id_{W(\varpi_r)_{x_j}}\right) \otimes R(x_i,x_{i+1}) \otimes \left(\bigotimes_{j>i+1}
 \id_{W(\varpi_r)_{x_j}}\right)\\
 R_{\gs'\gs}(x_1,\ldots,x_\ell) =& R_{\gs'}(x_{\gs(1)},\ldots,x_{\gs(\ell)}) \circ R_\gs(x_1,\ldots,x_\ell) \\
 &\ \text{for $\gs,\gs' \in \fS_\ell$ such that 
 $\ell(\gs'\gs) = \ell(\gs')+\ell(\gs)$}.
\end{align*}
Set 
\[ k = \begin{cases} (\ga_r,\ga_r)/2 & \text{if $\fg$ is of nontwisted affine type},\\
                     1 & \text{if $\fg$ is of twisted affine type} .
       \end{cases}
\]
Let $\gs_0 \in \fS_\ell$ denote the longest element, and define a map $R_\ell$ by
\begin{align*}
 R_\ell = &R_{\gs_0}(q^{k(\ell-1)},q^{k(\ell-3)},\ldots,q^{k(1-\ell)})\colon\\
 &W(\varpi_r)_{q^{k(\ell-1)}} \otimes \cdots\otimes W(\varpi_r)_{q^{k(1-\ell)}} \to W(\varpi_r)_{q^{k(1-\ell)}} \otimes \cdots
 \otimes W(\varpi_r)_{q^{k(\ell-1)}}.
\end{align*}
We denote the image of $R_\ell$ by $W^{r,\ell}$.
Then $W^{r,\ell}$ is a simple $U_q'(\fg)$-module by \cite[Theorem 9.2]{MR1890649}, and it is easily seen that the Drinfeld polynomials 
(see \cite{MR1357195} for the nontwisted case and \cite{zbMATH01202580} for the twisted case) of $W^{r,\ell}$ are 
\[ P_i(u) = \begin{cases} (1-a^\dag_rq^{1-\ell}_ru)(1-a^\dag_rq^{3-\ell}_ru)\cdots(1-a^\dag_rq^{\ell-1}_ru) & \text{if $i=r$},\\
                          1 & \text{otherwise},
            \end{cases}
\]
where $a_r^\dag \in \Q(q)$ is the element such that the Drinfeld polynomials of $W(\varpi_r)$ are 
\[ P_r(u) = 1-a^\dag_ru, \ \ \ P_j(u) = 1 \ (j\neq r).
\]
The module $W^{r,\ell}$ is called the \textit{Kirillov-Reshetikhin module} (KR module for short) associated with $r,\ell$.

Let us recall how to define a prepolarization on $W^{r,\ell}$.
We begin with recalling the following lemma.

\begin{Lem}[{\cite[Lemma 3.4.1]{MR1194953}}]\label{Lem:admissible_pairing}
 Let $M_j$ and $N_j$ $(j=1,2)$ be $U_q'(\fg)$-modules, and assume for $j =1,2$ that there exists 
 an admissible pairing $(\ , \ )_j$ between $M_j$ and $N_j$.
 Then the pairing $(\ , \ )$ between $M_1 \otimes M_2$ and $N_1 \otimes N_2$ defined by 
 \[ (u_1 \otimes u_2, v_1 \otimes v_2) = (u_1,v_1)_1 (u_2,v_2)_2 \ \ \text{for $u_j \in M_j$, $v_j \in N_j$}
 \]
 is admissible.
\end{Lem}

By the lemma and Proposition \ref{Prop:fund.rep.} (v), there exists an admissible pairing $(\ , \ )_0$
between $W(\varpi_r)_{q^{k(\ell-1)}} \otimes\cdots
\otimes W(\varpi_r)_{q^{k(1-\ell)}}$ and $W(\varpi_r)_{q^{k(1-\ell)}} \otimes\cdots \otimes W(\varpi_r)_{q^{k(\ell-1)}}$.
Now a bilinear form $(\ , \ )$ on $W^{r,\ell}$ is defined by
\[ (R_\ell u,R_\ell v) = (u,R_\ell v)_0
\]
for $u,v \in W(\varpi_r)_{q^{k(\ell-1)}} \otimes\cdots\otimes W(\varpi_r)_{q^{k(1-\ell)}}$.
We define a $U_q'(\fg)_{K_\Z}$-submodule $(W^{r,\ell})_{K_\Z}$ of $W^{r,\ell}$ by
\begin{align*}
 (W^{r,\ell})_{K_\Z}
 =R_\ell\big(U_q'(\fg)_{K_\Z}w \otimes \cdots \otimes U_q'(\fg)_{K_\Z}w\big)
   \cap \big(U_q'(\fg)_{K_\Z}w \otimes \cdots \otimes U_q'(\fg)_{K_\Z}w\big).
\end{align*}

\begin{Prop}[{\cite[Proposition 3.4.3]{MR1194953}}]\label{Prop:fund_properties_KR}\
 \begin{enumerate}
  \setlength{\parskip}{1pt} 
  \setlength{\itemsep}{1pt} 
  \item[(i)] $(\ , \ )$ is a nondegenerate prepolarization on $W^{r,\ell}$.
  \item[(ii)] $||R_\ell(w^{\otimes \ell})||^2 = 1$.
  \item[(iii)] $\big((W^{r,\ell})_{K_\Z}, (W^{r,\ell})_{K_\Z}\big) \subseteq K_\Z$.
 \end{enumerate}
\end{Prop}

\section{Main theorem}

\subsection{Statement of the main theorem}\label{Subsection:Main_Thm}

In the rest of this paper, we assume that $\fg$ is either of type $G_2^{(1)}$ or $D_4^{(3)}$, whose Dynkin diagram is as follows;
\begin{equation}\label{eq:Dynkin}
 (G_2^{(1)}) \ \ \xygraph{!~:{@3{->}}
    \circ ([]!{+(0,-.3)}{0}) -[r]
    \circ ([]!{+(0,-.3)} {1}) : [r]
    \circ ([]!{+(0,-.3)} {2})} \ \ \ \ \ \ \ \ 
   (D_4^{(3)}) \ \ \xygraph{!~:{@3{<-}}
    \circ ([]!{+(0,-.3)}{0}) -[r]
    \circ ([]!{+(0,-.3)} {1}) : [r]
    \circ ([]!{+(0,-.3)} {2})}
\end{equation}
Note that in both types $\fg_0$ is of type $G_2$, but \textit{the labeling of the simple roots are inverted}.

Our main theorem is the following.

\begin{Thm}\label{Thm:Main_Theorem}
  Assume that $\fg$ is either of type $G_2^{(1)}$ or $D_4^{(3)}$. For any $\ell \in \Z_{>0}$, the KR module $W^{2,\ell}$ has a crystal 
  pseudobase.
\end{Thm}

Since the existence of a crystal pseudobase of $W^{1,\ell}$ for any $\ell \in \Z_{>0}$ has already been proved
by \cite{zbMATH01251404} in type $G_2^{(1)}$ and by \cite{zbMATH05226997} in type $D_4^{(3)}$, 
this theorem implies that every KR module of these types has a crystal pseudobase.


\subsection{Reduction to calculations on the prepolarization}\label{Subsection:Reduction}

Hereafter, we write $W^{\ell}$ for $W^{2,\ell}$.
In order to apply Proposition \ref{Prop:criterion}, let us recall the branching rule of $W^{\ell}$.
There is an explicit formula for the decomposition of general KR modules as $U_q(\fg_0)$-modules, called $(q=1)$ \textit{fermionic formula},
see \cite{MR1745263,MR1903978,MR1993360,MR2254805,MR2428305,MR2576287}. 
For $W^{\ell}$ in type $G_2^{(1)}$ and $D_4^{(3)}$, the formulas are rewritten as follows.

\begin{Prop}\label{Prop:dec_G}
 For a positive integer $\ell$, define a set $S_\ell \subseteq \Z_{\ge 0}^{\times 4}$ by
 \begin{equation*}
  S_\ell = \{(a,b,c,d)\in \Z_{\ge 0}^{\times 4}\mid 3b \leq c \leq b+d, \ a\leq b,\ -c+3d\leq\ell\} \ \ \text{if $\fg$ is of type $G_2^{(1)}$},
 \end{equation*}
 and
 \begin{equation*}
  S_\ell = \{(a,b,c,d)\in \Z_{\ge 0}^{\times 4}\mid 3c \leq b+d, \ a\leq b \leq c,\ -c+d \leq\ell\} \ \ \text{if $\fg$ is of type $D_4^{(3)}$}.
 \end{equation*}
 Then we have
 \[ W^{\ell} \cong \bigoplus_{(a,b,c,d)\in S_\ell} V_0\big(\ell\varpi_2+(a+d)\ga_0+(b+d)\ga_1+c\ga_2\big)
 \]
 as $U_q(\fg_0)$-modules. 
\end{Prop}

\begin{proof}
 First assume that $\fg$ is of type $G_2^{(1)}$. 
 In \cite[Subsection 2.4]{MR2372556}, the fermionic formula of $W^{\ell}$ is rewritten as follows:
 \[ W^{\ell} \cong \bigoplus_{(r_1,r_2,r_3,r_4) \in T_\ell} V_0\big((r_2+r_3-r_4)\varpi_1+(\ell-r_1-3r_2-3r_3)\varpi_2\big),
 \]
 where $T_\ell \subseteq \Z_{\ge 0}^{\times 4}$ is defined by 
 \[ T_\ell= \{(r_1,r_2,r_3,r_4) \in \Z_{\ge 0}^{\times 4} \mid r_4 \leq r_2, \ 2r_1+3r_2+3r_3 \leq \ell\}.
 \]
 Then the assertion follows by the following substitution;
 \[ r_1 = -3b+c, \ \ \ r_2=a+b-c+d,\ \ \ r_3=-a+b,\ \ \ r_4=a,
 \]
 since $\ga_0= -\varpi_1$, $\ga_1=2\varpi_1 - 3\varpi_2$, $\ga_2 = -\varpi_1+2\varpi_2$.

 Next assume that $\fg$ is of type $D_4^{(3)}$. 
 In this case the branching rule given in \textit{loc.\ cit.} is as follows:
 \[ W^{\ell} \cong \bigoplus_{(r_1,r_2,r_3,r_4)\in T_\ell}V_0\big((r_1+r_2-r_3)\varpi_1+(\ell-r_1-r_2-r_4)\varpi_2\big)
 \]
 with $T_\ell = \{ (r_1,r_2,r_3,r_4)\in \Z_{\ge 0}^{\times 4} \mid r_3 \leq r_1, \ r_1+r_2+r_3+r_4 \leq \ell\}$.
 Then the assertion is proved via
 \[ r_1=-2c+d, \ \ \ r_2=-a+b,\ \ \ r_3=-b+c, \ \ \ r_4=a, 
 \]
 since $\ga_0=-\varpi_1$, $\ga_1=2\varpi_1-\varpi_2$, $\ga_2=-3\varpi_1+2\varpi_2$.
\end{proof}

For a quadruple $(a,b,c,d) \in \Z^{4}$ of integers,
set
\[ e^{(a,b,c,d)} = e_0^{(a)}e_1^{(b)}e_2^{(c)}e_1^{(d)}e_0^{(d)} \in U_q'(\fg),
\]
where we use the convention $e_i^{(k)} = 0$ for $k<0$.
Fix a vector $v_1 \in W(\varpi_r)$ with weight $\varpi_r$ and satisfying $(v_1,v_1) =1$ (which is denoted by $w$ in the previous section),
and for $\ell > 1$ let $v_\ell \in W^\ell$ denote the image of $v_1^{\otimes \ell}$ under $R_\ell$.
From Propositions \ref{Prop:criterion}, \ref{Prop:fund_properties_KR} and \ref{Prop:dec_G}, 
we easily see that the following proposition implies Theorem \ref{Thm:Main_Theorem}.

\begin{Prop}\label{Prop:hwv_G}
 Let $\ell$ be a positive integer, and $S_\ell \subseteq \Z_{\ge 0}^{\times 4}$ the set defined in Proposition \ref{Prop:dec_G}.
 Then the vectors $\{e^{(a,b,c,d)}v_\ell \mid (a,b,c,d) \in S_\ell\}$ in $W^{\ell}$ satisfy
 \begin{align}\label{eq:bilinear_value1}
  (e^{(a,b,c,d)}v_\ell,e^{(a',b',c',d')}v_\ell) &\in \gd_{(a,b,c,d),(a',b',c',d')} + qA,\\ \label{eq:bilinear_value2}
  ||e_ie^{(a,b,c,d)}v_\ell||^2 &\in q_i^{-2\langle h_i,\wt(e^{(a,b,c,d)}v_\ell)\rangle-2}q A \ \ \text{for } i=1,2.
 \end{align}
\end{Prop}

\subsection{Collection of equalities}\label{Subsection:equalities}

Here we collect equalities which are repeatedly used in the proof of Proposition \ref{Prop:hwv_G}.

The following equality is easily proved from (\ref{eq:coproduct}): for $i \in I$ and $k \in \Z_{\ge 0}$,
\begin{equation}\label{eq:div_coprod_e}
 \gD(e_i^{(k)}) = \sum_{j=0}^k q_i^{j(k-j)}e_i^{(k-j)} \otimes t_i^{-k+j}e_i^{(j)}.
\end{equation}
It is easily proved from the defining relations of $U_q'(\fg)$ that, for $i \in I$, $r,s\in \Z_{\ge 0}$,
\[ f_i^{(r)}e_i^{(s)} = \sum_{k=0}^{\min\{r,s\}} e_i^{(s-k)}f_i^{(r-k)} 
   \prod_{j=1}^{k} \frac{q_i^{r-s-j+1}t_i^{-1}-q_i^{-(r-s-j+1)}t_i}{q_i^j-q^{-j}_i}.
\]
In particular, if $M$ is a $U_q'(\fg)$-module we have
\begin{equation}\label{eq:commutator}
 f_i^{(r)}e_i^{(s)}v = \sum_{k=0}^{\min\{r,s\}} \begin{bmatrix} r-s-\langle h_i,\gl\rangle \\ k \end{bmatrix}_{q_i}e_i^{(s-k)}f_i^{(r-k)}v
 \ \ \ \text{for $v \in M_\gl$}.
\end{equation}
Assume that a given $U_q'(\fg)$-module $M$ has a prepolarization $(\ , \ )$.
Then for $i \in I$, $u,v \in M$ and $k \in \Z_{\ge 0}$, we have from (\ref{eq:admissible}) that
\[ (e_i^{(k)}u,v) = (u,q_i^{-k^2}t_i^{-k}f_i^{(k)}v),\ \ \ \ (f_i^{(k)}u,v)=(u, q_i^{-k^2}t_i^ke_i^{(k)}v).
\]
In particular, it holds that
\begin{equation}\label{eq:prepolarization_calculation}
 (e_i^{(k)}u,v) = q_i^{k(k-\langle h_i,\gl\rangle)}(u,f_i^{(k)}v),\ \ \ \ (f_i^{(k)}u,v)=q_i^{k(k+\langle h_i,\gl\rangle)}(u,e_i^{(k)}v)
   \ \ \ \text{for $v \in M_\gl$}.
\end{equation}
The following lemma is immediate from (\ref{eq:commutator}) and (\ref{eq:prepolarization_calculation}).

\begin{Lem}\label{Lem:Weyl_action}
 Let $M$ be a finite-dimensional $U_q'(\fg)$-module with a prepolarization $(\ ,\ )$, $u \in M_\gl$, $i \in I$ and assume that $f_iu=0$.
 Then
 \begin{align}\label{eq:Weyl_action}
  ||e_i^{(k)}u||^2=q_i^{-k(k+\langle h_i,\gl\rangle)}\begin{bmatrix}
  -\langle h_i,\gl\rangle\\k\end{bmatrix}_{q_i} ||u||^2.
 \end{align}
 In particular, $||e_i^{(k)}u||^2 \in (1+qA)||u||^2$ holds for $0\leq k \leq -\langle h_i, \gl \rangle.$
\end{Lem}

\subsection{Proof in type $G_2^{(1)}$ case}\label{Subsection:G_2^{(1)}}

Throughout this subsection we assume that $\fg$ is of type $G_2^{(1)}$,
and we will show Proposition \ref{Prop:hwv_G} in this type
(the type $D_4^{(3)}$ will be treated in the next subsection).
Note that 
\[ q_0=q_1=q^3,\ \ \ q_2=q.
\]

We write $e^{(b,c,d)}$ for $e^{(0,b,c,d)}=e_1^{(b)}e_2^{(c)}e_1^{(d)}e_0^{(d)}$.
For each $\ell \in \Z_{>0}$, let us consider the following collection of statements $\condC{1}{\ell}$--$\condC{6}{\ell}$ concerning 
the vector $v_\ell \in W^\ell$:
\begin{itemize}
 \item[$\condC{1}{\ell}$] $e^{(b,c,d)}v_\ell \neq 0$ if and only if $2b\leq c \leq 2d \leq 2\ell$.\\[-2pt]
 \item[$\condC{2}{\ell}$] $||e^{(b,c,d)}v_\ell||^2 \in 1+qA$ if $2b\leq c \leq 2d \leq 2\ell$.\\[-2pt]
 \item[$\condC{3}{\ell}$] For $(b,c,d),(b',c',d') \in \Z_{\ge 0}^{\times 3}$,
  $(e^{(b,c,d)}v_\ell,e^{(b',c',d')}v_\ell) =0$ if $(b,c,d) \neq (b',c',d')$.\\[-2pt]
 \item[$\condC{4}{\ell}$] For $(b,c,d) \in \Z_{\ge 0}^{\times 3}$, $||e_2e^{(b,c,d)}v_\ell||^2 \in q^{\min\{0,6b-2c,-2c+6d-2\ell\}}A$.\\[-2pt]
 \item[$\condC{5}{\ell}$] For $(b,c,d) \in \Z_{\ge 0}^{\times 3}$, $(e^{(b-1,c-2,d-1)}v_\ell,e_2e^{(b,c,d)}v_\ell) \in q^{-c+3d-\ell}A$.\\[-2pt]
 \item[$\condC{6}{\ell}$] For $(b,c,d) \in \Z_{\ge 0}^{\times 3}$, $(e^{(b,c,d)}v_\ell,e_2e^{(b,c-1,d)}v_\ell) \in q^{3b-c+1}A$.
\end{itemize}
In what follows, we will first show $\condC{1}{\ell}$--$\condC{6}{\ell}$ by the induction on $\ell$,
and then prove Proposition \ref{Prop:hwv_G} using them.

The fundamental module $W^{1}=W(\varpi_2)$ is isomorphic to $V_0(\varpi_2)$ as a $U_q(\fg_0)$-module 
and all the weight spaces are $1$-dimensional.
Then its $U_q'(\fg)$-module structure is automatically determined from the fact that $W(\varpi_2)$ has a global basis \cite{MR1890649}, 
and we describe the module structure explicitly here.
Let $W$ be the $7$-dimensional $\Q(q)$-vector space with a basis $\{\fbox{\raisebox{0pt}[7.5pt][0pt]{$j$}},
\fbox{\raisebox{0pt}[7.5pt][0pt]{$\ol{j}$}}\mid j=1,2,3\} \cup \{\fbox{\raisebox{0pt}[7.5pt][0pt]{$0$}}\}$.
Define a $U_q'(\fg)$-module structure on $W$ by
\begin{align*} 
 e_0\,\fbox{\raisebox{0pt}[7.5pt][0pt]{$1$}} &= \fbox{\raisebox{0pt}[7.5pt][0pt]{$\ol{2}$}}, \ 
 e_0\,\fbox{\raisebox{0pt}[7.5pt][0pt]{$2$}} = \fbox{\raisebox{0pt}[7.5pt][0pt]{$\ol{1}$}}, \text{ and } 
 e_0\,\fbox{\raisebox{1.5pt}[7.5pt][0pt]{$*$}} = 0 \text{ otherwise,}\\
 e_1\,\fbox{\raisebox{0pt}[7.5pt][0pt]{$3$}} &= \fbox{\raisebox{0pt}[7.5pt][0pt]{$2$}}, \ 
 e_1\,\fbox{\raisebox{0pt}[7.5pt][0pt]{$\ol{2}$}} = \fbox{\raisebox{0pt}[7.5pt][0pt]{$\ol{3}$}}, \text{ and }  
 e_1\,\fbox{\raisebox{1.5pt}[7.5pt][0pt]{$*$}} = 0 \text{ otherwise,}\\
 e_2\,\fbox{\raisebox{0pt}[7.5pt][0pt]{$2$}} &= \fbox{\raisebox{0pt}[7.5pt][0pt]{$1$}}, \ 
 e_2\,\fbox{\raisebox{0pt}[7.5pt][0pt]{$0$}} = [2]_{q}\fbox{\raisebox{0pt}[7.5pt][0pt]{3}}, \ 
 e_2\,\fbox{\raisebox{0pt}[7.5pt][0pt]{$\ol{3}$}} = \fbox{\raisebox{0pt}[7.5pt][0pt]{$0$}},\
 e_2\,\fbox{\raisebox{0pt}[7.5pt][0pt]{$\ol{1}$}} = \fbox{\raisebox{0pt}[7.5pt][0pt]{$\ol{2}$}}, \text{ and } 
 e_2\,\fbox{\raisebox{1.5pt}[7.5pt][0pt]{$*$}} = 0 \text{ otherwise,}\\
 f_0\,\fbox{\raisebox{0pt}[7.5pt][0pt]{$\ol{2}$}}&=\fbox{\raisebox{0pt}[7.5pt][0pt]{$1$}}, \ 
 f_0\,\fbox{\raisebox{0pt}[7.5pt][0pt]{$\ol{1}$}} = \fbox{\raisebox{0pt}[7.5pt][0pt]{$2$}}, \text{ and } 
 f_0\,\fbox{\raisebox{1.5pt}[7.5pt][0pt]{$*$}} = 0 \text{ otherwise,}\\
 f_1\,\fbox{\raisebox{0pt}[7.5pt][0pt]{$2$}} &= \fbox{\raisebox{0pt}[7.5pt][0pt]{$3$}}, \ 
 f_1\,\fbox{\raisebox{0pt}[7.5pt][0pt]{$\ol{3}$}} = \fbox{\raisebox{0pt}[7.5pt][0pt]{$\ol{2}$}}, \text{ and }  
 f_1\,\fbox{\raisebox{1.5pt}[7.5pt][0pt]{$*$}} = 0 \text{ otherwise.}\\
 f_2\,\fbox{\raisebox{0pt}[7.5pt][0pt]{$1$}}&= \fbox{\raisebox{0pt}[7.5pt][0pt]{$2$}}, \ 
 f_2\,\fbox{\raisebox{0pt}[7.5pt][0pt]{$3$}}=\fbox{\raisebox{0pt}[7.5pt][0pt]{$0$}}, \ 
 f_2\,\fbox{\raisebox{0pt}[7.5pt][0pt]{$0$}} = [2]_q\fbox{\raisebox{0pt}[7.5pt][0pt]{$\ol{3}$}}, \ 
 f_2\,\fbox{\raisebox{0pt}[7.5pt][0pt]{$\ol{2}$}} = \fbox{\raisebox{0pt}[7.5pt][0pt]{$\ol{1}$}}, \text{ and } 
 f_2\,\fbox{\raisebox{1.5pt}[7.5pt][0pt]{$*$}}=0 \text{ otherwise,}\\
 \wt(\fbox{\raisebox{0pt}[7.5pt][0pt]{$j$}})&=\gd_{j1}\varpi_2+\gd_{j2}(\varpi_1-\varpi_2)+\gd_{j3}(-\varpi_1+2\varpi_2),\ 
 \wt(\fbox{\raisebox{0pt}[7.5pt][0pt]{$\ol{j}$}}) = -\wt(\fbox{\raisebox{0pt}[7.5pt][0pt]{$j$}}).
\end{align*}
Then we have $W^{1} \cong W$ and identify $W^{1}$ with $W$ hereafter.
To illustrate the actions, we give the crystal graph of $W^{1}$ below:
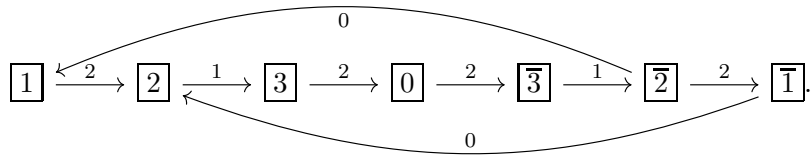
\[ \begin{tikzcd}
    \fbox{\raisebox{0pt}[7.5pt][0pt]{$1$}} \arrow[r, "2"]&\fbox{\raisebox{0pt}[7.5pt][0pt]{$2$}} \arrow[r, "1"] &
    \fbox{\raisebox{0pt}[7.5pt][0pt]{$3$}} \arrow[r,"2"]& \fbox{\raisebox{0pt}[7.5pt][0pt]{$0$}}\arrow[r,"2"] &
    \fbox{\raisebox{0pt}[7.5pt][0pt]{$\ol{3}$}} \arrow[r,"1"] &
    \fbox{\raisebox{0pt}[7.5pt][0pt]{$\ol{2}$}} \arrow[lllll, bend right = 23, "0"] \arrow[r,"2"] & 
    \fbox{\raisebox{0pt}[7.5pt][0pt]{$\ol{1}$}}\arrow[lllll, bend left=20, "0"'].
   \end{tikzcd}
\]
Using Lemma \ref{Lem:Weyl_action}, we easily check that
\begin{equation}\label{bilinear_values}
 ||\,\fbox{\raisebox{0pt}[7.5pt][0pt]{$j$}}\,||^2 = 1 \ \ \text{for } j\in\{1,2,3,\ol{1},\ol{2},\ol{3}\},\ \ \ 
 ||\,\fbox{\raisebox{0pt}[7.5pt][0pt]{$0$}}\,||^2 = 1+q^2.
\end{equation}
Now $\condC{1}{1}$--$\condC{6}{1}$ are easily checked, and the induction begins.

Let $\ell > 1$, and assume that $\condC{1}{\ell-1}$--$\condC{6}{\ell-1}$ hold.
$W^{\ell-1}$ being a submodule of $W^1_{q^{2-\ell}} \otimes \cdots \otimes W^1_{q^{\ell-2}}$, 
there exists an injection from $(W^{\ell-1})_{q^{-1}}$ to
\[ \Big((W^1)_{q^{2-\ell}} \otimes \cdots \otimes (W^1)_{q^{\ell-2}}\Big)_{q^{-1}} \cong (W^1)_{q^{1-\ell}}\otimes \cdots
   \otimes (W^1)_{q^{\ell-3}}.
\]
Hence we have an injection
\[ (W^{\ell-1})_{q^{-1}} \otimes (W^1)_{q^{\ell-1}} \hookrightarrow (W^1)_{q^{1-\ell}}\otimes \cdots \otimes (W^1)_{q^{\ell-3}}\otimes 
   (W^1)_{q^{\ell-1}}.
\]
By the simplicity of $W^\ell$, we see that $W^{\ell}$ is isomorphic to the submodule of 
$(W^{\ell-1})_{q^{-1}} \otimes (W^1)_{q^{\ell-1}}$ generated by $(v_{\ell-1})_{q^{-1}} \otimes (v_1)_{q^{\ell-1}}$.
Hence for $(b,c,d) \in \Z_{\ge 0}^{\times 3}$, 
\begin{align}\label{eq:nonzero_condition}
  e^{(b,c,d)}v_\ell \neq 0 \ \ &\text{if and only if} \ \ e^{(b,c,d)}\big((v_{\ell-1})_{q^{-1}} \otimes (v_1)_{q^{\ell-1}}\big) \neq 0
   \text{ in } (W^{\ell-1})_{q^{-1}} \otimes (W^1)_{q^{\ell-1}}\nonumber\\
  &\text{if and only if} \ \ e^{(b,c,d)}\big(v_{\ell-1} \otimes (v_1)_{q^\ell}\big) \neq 0 \text{ in } W^{\ell-1} \otimes (W^1)_{q^\ell}.
\end{align}
The following lemma is proved by straightforward calculations using (\ref{eq:div_coprod_e}) (note that we have
$e_1^{(k)}e_0^{(m)}v_{\ell-1}=0$
for $k>m$ by the Serre relations).

\begin{Lem}\label{Lem:direct_calc_coprod}
 For any $m \in \Z$, we have
 \begin{align*}
  e^{(b,c,d)}\big(v_{\ell-1} \otimes (v_1)_{q^m}\big) &= q^{-c+3d}e^{(b,c,d)}v_{\ell-1}
   \otimes \fboxr{1} + q^{m}e^{(b-1,c-2,d-1)}v_{\ell-1} \otimes \fboxr{2}\\
   &+ q^{3b+m}e^{(b,c-2,d-1)}v_{\ell-1}\otimes \fboxr{3} + q^{c+m-1}e^{(b,c-1,d-1)}v_{\ell-1}\otimes \fboxr{0}\\
   &+q^{-3b+2c+m}e^{(b,c,d-1)}v_{\ell-1}\otimes \fboxr{\ol{3}}.
 \end{align*}
\end{Lem}

Now $\condC{1}{\ell}$ follows from this lemma, (\ref{eq:nonzero_condition}) and $\condC{1}{\ell-1}$.
Next we will show $\condC{2}{\ell}$.
By the definition, the composition of the maps
\begin{align*}
 (W^1)_{q^{\ell-1}} \otimes \cdots \otimes (W^1)_{q^{1-\ell}} &\stackrel{R_{\ell-1}'\otimes \id_{(W^1)_{q^{1-\ell}}}}{\to}
 (W^1)_{q^{3-\ell}}\otimes \cdots \otimes (W^1)_{q^{\ell-1}} \otimes (W^1)_{q^{1-\ell}}\\ &\stackrel{R_{s_1s_2\cdots s_{\ell-1}}}{\to}
 (W^1)_{q^{1-\ell}} \otimes \cdots \otimes (W^1)_{q^{\ell-1}}
\end{align*}
coincides with $R_\ell$, where 
\[R_{\ell-1}' \colon (W^1)_{q^{\ell-1}}\otimes \cdots \otimes (W^1)_{q^{3-\ell}} \to 
  (W^1)_{q^{3-\ell}} \otimes \cdots \otimes (W^1)_{q^{\ell-1}}
\]
is defined similarly to $R_{\ell-1}$.
Since the image of the first map is $(W^{\ell-1})_q \otimes (W^1)_{q^{1-\ell}}$ and that of $R_\ell$ is $W^{\ell}$,
the second map induces a $U_q'(\fg)$-module homomorphism
\[ R' \colon (W^{\ell-1})_q \otimes (W^1)_{q^{1-\ell}} \to W^\ell.
\]
Let $(\ , \ )_1$ be the admissible pairing between $(W^{\ell-1})_q \otimes (W^1)_{q^{1-\ell}}$ and $(W^{\ell-1})_{q^{-1}} \otimes 
(W^1)_{q^{\ell-1}}$ which is obtained by Lemma \ref{Lem:admissible_pairing}.
Now the following lemma is obvious from the construction of the prepolarization $(\ , \ )$ on $W^\ell$.

\begin{Lem}\label{Lem:bilinear_induction}
 For $u, v \in (W^{\ell-1})_q \otimes (W^1)_{q^{1-\ell}}$, we have
 \[ \big(R'(u),R'(v)\big) = \big(u,R'(v)\big)_1,
 \]
 where $R'(v)$ in the right-hand side is regarded as a vector of $(W^{\ell-1})_{q^{-1}} \otimes 
(W^1)_{q^{\ell-1}}$ via the embedding $W^\ell \hookrightarrow (W^{\ell-1})_{q^{-1}} \otimes (W^1)_{q^{\ell-1}}$ mapping $v_\ell$ to 
$(v_{\ell-1})_{q^{-1}} \otimes (v_1)_{q^{\ell-1}}$.
\end{Lem}

Denote by $(\ , \ )_2$ the admissible pairing between $W^{\ell-1}\otimes (W^1)_{q^\ell}$ and $W^{\ell-1}\otimes (W^1)_{q^{-\ell}}$.
For $(b,c,d) \in \Z_{\ge 0}^{\times 3}$, we have from Lemma \ref{Lem:bilinear_induction} that
\begin{align*}
  ||e^{(b,c,d)}v_\ell||^2 &= \Big(e^{(b,c,d)}\big((v_{\ell-1})_q \otimes (v_1)_{q^{1-\ell}}\big), e^{(b,c,d)}\big((v_{\ell-1})_{q^{-1}}
   \otimes (v_1)_{q^{\ell-1}}\big)\Big)_1\\
   &= \Big(e^{(b,c,d)}\big(v_{\ell-1} \otimes (v_1)_{q^{-\ell}}\big), e^{(b,c,d)}\big(v_{\ell-1}
   \otimes (v_1)_{q^{\ell}}\big)\Big)_2,
\end{align*}
and the right-hand side is equal to 
\begin{align*}
 q^{-2c+6d}||e^{(b,c,d)}v_{\ell-1}||^2 &+ ||e^{(b-1,c-2,d-1)}v_{\ell-1}||^2+q^{6b}||e^{(b,c-2,d-1)}v_{\ell-1}||^2\\
  &+q^{2c-1}[2]_q||e^{(b,c-1,d-1)}v_{\ell-1}||^2 + q^{-6b+4c}||e^{(b,c,d-1)}v_{\ell-1}||^2
\end{align*}
by Lemma \ref{Lem:direct_calc_coprod} and (\ref{bilinear_values}).
Now $\condC{2}{\ell}$ is easily proved from $\condC{2}{\ell-1}$.
$\condC{3}{\ell}$ is shown in a similar manner by calculating $(e^{(b,c,d)}v_\ell,e^{(b',c',d')}v_\ell)$ 
using Lemmas \ref{Lem:direct_calc_coprod},
\ref{Lem:bilinear_induction} and applying $\condC{3}{\ell-1}$.

Next we will show $\condC{4}{\ell}$. 
By $\condC{1}{\ell}$ we may assume $2b \leq c \leq 2d \leq 2\ell$.
When $b=0$, $\condC{2}{\ell}$ implies $||e_2e^{(0,c,d)}v_\ell||^2 \in q^{-2c}A$, and $\condC{4}{\ell}$ holds.
Hence we may further assume that $b > 0$.
Writing $v =v_{\ell-1}$ for short, we have from Lemma \ref{Lem:direct_calc_coprod} that
\begin{align}
 e_2&e^{(b,c,d)}\big(v \otimes (v_1)_{q^m}\big)\nonumber\\
 = &(q^{-c+3d-1}e_2e^{(b,c,d)}v +q^me^{(b-1,c-2,d-1)}v)\otimes \fboxr{1}
 +q^{m+1}e_2e^{(b-1,c-2,d-1)}v\otimes \fboxr{2}\nonumber\\
 +& q^m(q^{3b-2}e_2e^{(b,c-2,d-1)}v+q^{c-1}[2]_{q}e^{(b,c-1,d-1)}v)\otimes  \fboxr{3}\label{eq:direct_cal2}\\
 +&q^m(q^{c-1}e_2e^{(b,c-1,d-1)}v+q^{-3b+2c}e^{(b,c,d-1)}v)\otimes \fboxr{0}
 + q^{-3b+2c+m+2}e_2e^{(b,c,d-1)}v\otimes \fboxr{\ol{3}},\nonumber
\end{align}
and hence we have from Lemma \ref{Lem:bilinear_induction} that
\begin{align}
 &||e_2e^{(b,c,d)}v_\ell||^2=q^{-c+3d-1}(q^\ell+q^{-\ell})(e^{(b-1,c-2,d-1)}v,e_2e^{(b,c,d)}v)\label{eq:e2-bilinear1}\\
 &+||e^{(b-1,c-2,d-1)}v||^2+q^{2c-2}[2]_q^2||e^{(b,c-1,d-1)}v||^2+q^{-6b+4c+1}[2]_q||e^{(b,c,d-1)}v||^2\label{eq:e2-bilinear2}\\
 &+ 2[2]_q\Big(q^{3b+c-3}(e^{(b,c-1,d-1)}v,e_2e^{(b,c-2,d-1)}v)+q^{-3b+3c}(e^{(b,c,d-1)}v,e_2e^{(b,c-1,d-1)}v)\Big)\label{eq:e2-bilinear3}\\
 &\begin{aligned}
 &+q^{-2c+6d-2}||e_2e^{(b,c,d)}v||^2+q^2||e_2e^{(b-1,c-2,d-1)}v||^2+q^{6b-4}||e_2e^{(b,c-2,d-1)}v||^2\\
 &\hspace{100pt}+q^{2c-1}[2]_q||e_2e^{(b,c-1,d-1)}v||^2+q^{-6b+4c+4}||e_2e^{(b,c,d-1)}v||^2.
 \end{aligned}\label{eq:e2-bilinear4}
\end{align}
We need to show that all (\ref{eq:e2-bilinear1})--(\ref{eq:e2-bilinear4}) 
belong to $q^{\min\{0,6b-2c,-2c+6d-2\ell\}}A$.
$\condC{5}{\ell-1}$ implies $\text{(\ref{eq:e2-bilinear1})} \in q^{-2c+6d-2\ell}A$,
and $\condC{2}{\ell-1}$ implies $\text{(\ref{eq:e2-bilinear2})} \in A$ since we are assuming $2\leq 2b\leq c$.
It is easily proved from $\condC{6}{\ell-1}$ that $\text{(\ref{eq:e2-bilinear3})} \in A$.
Finally let us show that $\text{(\ref{eq:e2-bilinear4})} \in q^{\min\{0,6b-2c,-2c+6d-2\ell\}}A$.
Since $2\leq c\leq 2d$, it follows from $\condC{4}{\ell-1}$ that
\[ q^{-2c+6d-2}||e_2e^{(b,c,d)}v||^2 \in q^{c-2}q^{\min\{0,6b-2c,-2c+6d-2\ell+2\}}A \subseteq q^{\min\{0,6b-2c,-2c+6d-2\ell\}}A.
\]
It is proved by the same argument that all the other terms of (\ref{eq:e2-bilinear4}) also belong to $q^{\min\{0,6b-2c,-2c+6d-2\ell\}}A$,
and hence $\condC{4}{\ell}$ holds.

The proof of $\condC{5}{\ell}$ is similar. Writing $v=v_{\ell-1}$, 
we have from Lemmas \ref{Lem:direct_calc_coprod}, \ref{Lem:bilinear_induction}, and (\ref{eq:direct_cal2}) that
\begin{align*}
 (&e^{(b-1,c-2,d-1)}v_\ell,e_2e^{(b,c,d)}v_\ell)\\
  &= \Big(e^{(b-1,c-2,d-1)}\big((v)_q\otimes (v_1)_{q^{1-\ell}}\big),e_2e^{(b,c,d)}
  \big((v)_{q^{-1}} \otimes (v_1)_{q^{\ell-1}}\big)\Big)_1\\
 &=q^{-1}\Big(e^{(b-1,c-2,d-1)}\big(v\otimes (v_1)_{q^{-\ell}}\big),e_2e^{(b,c,d)}\big(v\otimes (v_1)_{q^{\ell}}\big)\Big)_2\\
 &=q^{-2c+6d-3}(e^{(b-1,c-2,d-1)}v,e_2e^{(b,c,d)}v)+q^{-c+3d+\ell-2}||e^{(b-1,c-2,d-1)}v||^2\\
 &+(e^{(b-2,c-4,d-2)}v,e_2e^{(b-1,c-2,d-1)}v)+q^{6b-6}(e^{(b-1,c-4,d-2)}v,e_2e^{(b,c-2,d-1)}v)\\
 &+q^{3b+c-5}[2]_q(e^{(b-1,c-4,d-2)}v,e^{(b,c-1,d-1)}v)+q^{2c-4}[2]_q(e^{(b-1,c-3,d-2)}v,e_2e^{(b,c-1,d-1)}v)\\
 &+q^{-3b+3c-3}[2]_q(e^{(b-1,c-3,d-2)}v,e^{(b,c,d-1)}v))+q^{-6b+4c}(e^{(b-1,c-2,d-2)}v,e_2e^{(b,c,d-1)}v).
\end{align*}
All the terms are easily seen to belong to $q^{-c+3d-\ell}A$ by applying $\condC{1}{\ell-1}$--$\condC{3}{\ell-1}$ and $\condC{5}{\ell-1}$,
and hence $\condC{5}{\ell}$ holds.
The proof of $\condC{6}{\ell}$ is similar, and we omit the details.

Now let us prove Proposition \ref{Prop:hwv_G} by using $\condC{1}{\ell}$--$\condC{4}{\ell}$.
Let $(a,b,c,d), (a',b',c',d') \in S_\ell$, and assume first that $a\neq a'$.
We may assume $a>a'$, and then by (\ref{eq:prepolarization_calculation}) we have
\[ (e^{(a,b,c,d)}v_\ell,e^{(a',b',c',d')}v_\ell) = q^{3a(a-2a'+b'-d'+\ell)}(e^{(b,c,d)}v_\ell,f_0^{(a)}e^{(a',b',c',d')}v_\ell) = 0,
\]
since $f_0v_\ell = 0$ and $e_1^{(d')}e_0^{(d'')}v_\ell=0$ if $d''<d'$.
Therefore (\ref{eq:bilinear_value1}) holds in this case.
On the other hand, if $a=a'$, it follows from (\ref{eq:commutator}) and (\ref{eq:prepolarization_calculation}) that
\begin{align*}
 (e^{(a,b,c,d)}v_\ell,e^{(a',b',c',d')}v_\ell)&=q^{3a(-a+b'-d'+\ell)}(e^{(b,c,d)}v_\ell,f_0^{(a)}e^{(a',b',c',d')}v_\ell)\\
                                              &=q^{3a(-a+b'-d'+\ell)}\begin{bmatrix} b'-d'+\ell \\ a\end{bmatrix}_{q^3}
 (e^{(b,c,d)}v_\ell,e^{(b',c',d')}v_\ell),
\end{align*}
and hence (\ref{eq:bilinear_value1}) holds by $\condC{2}{\ell}$ and $\condC{3}{\ell}$.

Assume that $(a,b,c,d) \in S_\ell$. By Lemma \ref{Lem:Weyl_action} and $\condC{4}{\ell}$, we have
\begin{align*}
 ||e_2e^{(a,b,c,d)}v_\ell||^2 = ||e_0^{(a)}e_2e^{(b,c,d)}v_\ell||^2 \in (1+qA)||e_2e^{(b,c,d)}v_\ell||^2 
  \subseteq q^{\min\{6b-2c,-2c+6d-2\ell\}}A,
\end{align*}
which implies (\ref{eq:bilinear_value2}) with $i = 2$. 
Finally it remains to show (\ref{eq:bilinear_value2}) with $i = 1$, namely,
\begin{equation*}
 ||e_1e^{(a,b,c,d)}v_\ell||^2 \in q^{6(a-2b+c-d-1)}qA, 
\end{equation*} 
which needs a little bit more calculations. 
By applying (\ref{eq:prepolarization_calculation}) and (\ref{eq:commutator}), we have
\begin{align*}
 ||&e_1e^{(a,b,c,d)}v_\ell||^2= q^{3(a-2b+c-d-1)}(e^{(a,b,c,d)}v_\ell, f_1e_1e^{(a,b,c,d)}v_\ell)\nonumber \\
 &= q^{3(a-2b+c-d-1)}\Big([a-2b+c-d]_{q^3}||e^{(a,b,c,d)}v_\ell||^2 + (e^{(a,b,c,d)}v_\ell, e_1f_1e^{(a,b,c,d)}v_\ell)\Big)\nonumber \\
 &\equiv q^{6(a-2b+c-d)}||f_1e^{(a,b,c,d)}v_\ell||^2 \ \ (\text{mod } q^{6(a-2b+c-d-1)}qA).
\end{align*}
Hence it suffices to show that $||f_1e^{(a,b,c,d)}v_\ell||^2 \in A$.
Set $p=b-d+\ell$. Note that $(a,b,c,d) \in S_\ell$ implies $p>0$. 
Since $f^{(k)}_0f_1e^{(b,c,d)}v_\ell=0$ ($k>1$) by the Serre relations, we have
\begin{align}
 ||&f_1e^{(a,b,c,d)}v_\ell||^2=q^{3a(p-a-1)}(f_1e^{(b,c,d)}v_\ell, f_0^{(a)}e_0^{(a)}f_1e^{(b,c,d)}v_\ell)\nonumber\\
 &=q^{3a(p-a-1)}\Big(\begin{bmatrix} p-1\\ a\end{bmatrix}_{q^3}||f_1e^{(b,c,d)}v_\ell||^2
 + \begin{bmatrix} p-1\\ a-1\end{bmatrix}_{q^3}(f_1e^{(b,c,d)}v_\ell, e_0f_0f_1e^{(b,c,d)}v_\ell)\Big)\nonumber \\
 &\in ||f_1e^{(b,c,d)}v_\ell||^2A+q^{6(p-a)}||f_0f_1e^{(b,c,d)}v_\ell||^2A.\label{eq:f1e}
\end{align}
It follows that
\begin{align*}
 ||&f_1e^{(b,c,d)}v_\ell||^2 = q^{6b-3c+3d-3}(e^{(b,c,d)}v_\ell,e_1f_1e^{(b,c,d)}v_\ell)\\
   &=q^{6b-3c+3d-3}\Big([2b-c+d]_{q^3}||e^{(b,c,d)}v_\ell||^2+(e^{(b,c,d)}v_\ell,f_1e_1e^{(b,c,d)}v_\ell)\Big)\\
   &=q^{6b-3c+3d-3}[2b-c+d]_{q^3}||e^{(b,c,d)}v_\ell||^2+q^{12b-6c+6d}[b+1]_{q^3}^2||e^{(b+1,c,d)}v_\ell||^2\in A.
\end{align*}
Moreover, since
\[ f_0f_1e^{(b,c,d)}v_\ell = f_0 \Big(e_1^{(b)}e_2^{(c)}e_1^{(d-1)}e_0^{(d)}v_\ell + \ga e^{(b-1,c,d)}v_\ell\Big)
   = [\ell-d+1]_{q^3}e^{(b,c,d-1)}v_\ell
\]
where $\ga \in \C(q)$ is a certain element, we have 
\[ q^{6(p-a)}||f_0f_1e^{(b,c,d)}v_\ell||^2 =q^{6(p-a)}[\ell-d+1]_{q^3}^2||e^{(b,c,d-1)}v_\ell||^2\in  A. 
\]
Now (\ref{eq:f1e}) implies $||f_1e^{(a,b,c,d)}v_\ell||^2 \in A$, as required.
The proof of Proposition \ref{Prop:hwv_G} is complete.

\subsection{Proof in type $D_4^{(3)}$ case}\label{Subsection:D_4^{(3)}}

Throughout this subsection we assume that $\fg$ is of type $D_4^{(3)}$,
and we will show Proposition \ref{Prop:hwv_G} in this type.
A large part of the proof is parallel to the case of $G_2^{(1)}$.
Note that 
\[ q_0=q_1=q, \ \ \ q_2=q^3.
\]

Let us consider the following collection of statements:
\begin{itemize}
 \item[$\condD{1}{\ell}$] $e^{(b,c,d)}v_\ell \neq 0$ if and only if $2b\leq 3c \leq 2d \leq 6\ell$.\\[-2pt]
 \item[$\condD{2}{\ell}$] $||e^{(b,c,d)}v_\ell||^2 \in 1+qA$ if $2b\leq 3c \leq 2d \leq 6\ell$.\\[-2pt]
 \item[$\condD{3}{\ell}$] For $(b,c,d),(b',c',d') \in \Z_{\ge 0}^{\times 3}$,
  $(e^{(b,c,d)}v_\ell,e^{(b',c',d')}v_\ell) =0$ if $(b,c,d) \neq (b',c',d')$.\\[-2pt]
 \item[$\condD{4}{\ell}$] For $(b,c,d) \in \Z_{\ge 0}^{\times 3}$, $||e_2e^{(b,c,d)}v_\ell||^2 \in 
  \begin{cases} q^{\min\{0,6b-6c\}}A & (b<3),\\
                q^{\min\{0,6b-6c,-6c+6d-6\ell\}}A & (b\geq 3).
  \end{cases}$\\
 \item[$\condD{5}{\ell}$] For $(b,c,d) \in \Z_{\ge 0}^{\times 3}$, $(e^{(b-3,c-2,d-3)}v_\ell,e_2e^{(b,c,d)}v_\ell) \in q^{-3c+3d-3\ell}A$.
  \\[-2pt]
 \item[$\condD{6}{\ell}$] For $(b,c,d) \in \Z_{\ge 0}^{\times 3}$, $(e^{(b,c,d)}v_\ell,e_2e^{(b,c-1,d)}v_\ell) \in q^{3b-3c+3}A$.
\end{itemize}

First let us show $\condD{1}{1}$--$\condD{6}{1}$ (unlike $G_2^{(1)}$ case, we do not give the explicit module structure here).
Write $w=v_1$.
$\condD{2}{1}$ follows from Lemma \ref{Lem:Weyl_action} except for the case $(b,c,d)=(1,1,3)$,
in which $\condD{2}{1}$ is proved by the following calculation:
\begin{align*}
 ||e^{(1,1,3)}w||^2&= q^{-1}(e^{(0,1,3)}w,f_1e^{(1,1,3)}w)=q^{-1}(e^{(0,1,3)}w, e_1e_2e_1^{(2)}e_0^{(3)}w)\\
  &=(f_1e^{(0,1,3)}w,e_2e_1^{(2)}e_0^{(3)}w)= ||e_2e_1^{(2)}e_0^{(3)}w||^2 \in 1+qA.
\end{align*}
Then to show $\condD{1}{1}$, it suffices to check that $e^{(b,c,d)}w = 0$ unless $2b\leq 3c \leq 2d \leq 6$.
By the Serre relations we have
\begin{align*}
 e^{(2,1,2)}w &= e_1^{(2)}e_2e_1^{(2)}e_0^{(2)}w=\sum_{k=0,1,3,4}
   (-1)^{k+1}e_1^{(k)}e_2e_1^{(4-k)}e_0^{(2)}w=e_1^{(3)}e_2e_1e_0^{(2)}w\\
 &= e_1^{(3)}e_2e_0e_1e_0w-e_1^{(3)}e_2e_0^{(2)}e_1w=0,
\end{align*}
and in all the other cases $e^{(b,c,d)}w=0$ is proved from Proposition \ref{Prop:fund.rep.} (iii) and $e_1^{(k)}e_0^{(m)}w=0$ for $k>m$.
Hence $\condD{1}{1}$ holds.
In order to show $\condD{3}{1}$, it is enough to check by the weight consideration that 
\begin{align*}
  (e^{(0,0,1)}w,e^{(1,1,2)}w) &= (e^{(0,0,1)}w,e^{(2,2,3)}w)=(e^{(1,1,2)}w,e^{(2,2,3)}w)\\
  &=(e^{(0,0,2)}w,e^{(1,1,3)}w) =(e^{(0,1,2)}w,e^{(1,2,3)}w)=0,
\end{align*}
and these are all proved by applying (\ref{eq:prepolarization_calculation}).
The statements $\condD{4}{1}$--$\condD{6}{1}$ are all proved from the following claims:
\begin{align*}
 \text{(a)} \ e_2e^{(b,c,d)}w&=0 \text{ unless $b=0$ or $(b,c,d) \in \{(1,1,3),(3,2,3)\}$},\\
 \text{(b)} \ e_2e^{(1,1,3)}w&=e^{(1,2,3)}w, \ \ \ \text{(c)} \ e_2e^{(3,2,3)}w=w.
\end{align*}
The claim (a) follows from Proposition \ref{Prop:fund.rep.} (iii), (b) from the Serre relations,
and (c) from the construction of $W^1=W(\varpi_2)$ (see \cite[Section 5]{MR1890649}).
Now the proofs of $\condD{1}{1}$--$\condD{6}{1}$ are complete.

Let $\ell >1$, and assume that $\condD{1}{\ell-1}$--$\condD{6}{\ell-1}$ hold.
The following lemma is proved by straightforward calculations.

\begin{Lem}\label{Lem:direct_D}
 Let $m \in \Z$, and write $v=v_{\ell-1}$, $w=v_1$. 
 Then the following equalities hold, where we simply write $e^{(b,c,d)}w$ for $(e^{(b,c,d)}w)_{q^m}$ in the right-hand sides.
 \begin{align*}
  \text{{\normalfont(i)}}& \ 
  e^{(b,c,d)}\big(v \otimes (w)_{q^m}\big) = q^{-3c+3d}e^{(b,c,d)}v\otimes w + q^{-b+2d+m-2}e^{(b,c,d-1)}v \otimes e^{(0,0,1)}w\\
  & + q^{-2b+3c+d+2m-2}e^{(b,c,d-2)}v\otimes e^{(0,0,2)}w + q^{b+d+2m-2}e^{(b,c-1,d-2)}v\otimes e^{(0,1,2)}w\\
  & + q^{d+2m-2}e^{(b-1,c-1,d-2)}v\otimes e^{(1,1,2)}w + q^{-3b+6c+3m}e^{(b,c,d-3)}v \otimes e^{(0,0,3)}w\\
  & + q^{3c+3m-3}e^{(b,c-1,d-3)}v \otimes e^{(0,1,3)}w + q^{-b+3c+3m-2}e^{(b-1,c-1,d-3)}v \otimes e^{(1,1,3)}w\\
  &+ q^{3b+3m}e^{(b,c-2,d-3)}v \otimes e^{(0,2,3)}w + q^{2b+3m-2}e^{(b-1,c-2,d-3)}v \otimes e^{(1,2,3)}w\\
  &+ q^{b+3m-2}e^{(b-2,c-2,d-3)}v \otimes e^{(2,2,3)}w+ q^{3m}e^{(b-3,c-2,d-3)}v \otimes e^{(3,2,3)}w.
 \end{align*}
 \begin{align*}
  \text{\normalfont{(ii)}}& \ 
   e_2e^{(b,c,d)}\big(v \otimes (w)_{q^m}\big)=(q^{-3c+3d-3}e_2e^{(b,c,d)}v+q^{3m}e^{(b-3,c-2,d-3)}v) \otimes w\\
   &+q^{-b+2d+m-2}e_2e^{(b,c,d-1)}v \otimes e^{(0,0,1)}w + q^{-2b+3c+d+2m+1}e_2e^{(b,c,d-2)}v\otimes e^{(0,0,2)}w\\
   &+ (q^{-2b+3c+d+2m-2}e^{(b,c,d-2)}v+q^{b+d+2m-5}e_2e^{(b,c-1,d-2)}v)\otimes e^{(0,1,2)}w\\
   &+ q^{d+2m-2}e_2e^{(b-1,c-1,d-2)}v\otimes e^{(1,1,2)}w+ q^{-3b+6c+3m+6}e_2e^{(b,c,d-3)}v \otimes e^{(0,0,3)}w\\
   &+ (q^{-3b+6c+3m}e^{(b,c,d-3)}v+ q^{3c+3m-3}e_2e^{(b,c-1,d-3)}v)\otimes e^{(0,1,3)}w\\
   &+ q^{-b+3c+3m+1}e_2e^{(b-1,c-1,d-3)}v \otimes e^{(1,1,3)}w\\
   & + (q^{3c+3m-3}[2]_{q^3}e^{(b,c-1,d-3)}v + q^{3b+3m-6}e_2e^{(b,c-2,d-3)}v) \otimes e^{(0,2,3)}w\\
   & + (q^{-b+3c+3m-2}e^{(b-1,c-1,d-3)}v + q^{2b+3m-5}e_2e^{(b-1,c-2,d-3)}v) \otimes e^{(1,2,3)}w\\
   & + q^{b+3m-2}e_2e^{(b-2,c-2,d-3)}v \otimes e^{(2,2,3)}w + q^{3m+3}e_2e^{(b-3,c-2,d-3)}v \otimes e^{(3,2,3)}w.
 \end{align*}
\end{Lem}

By the same arguments, (\ref{eq:nonzero_condition}) and Lemma \ref{Lem:bilinear_induction} are also proved in type $D_4^{(3)}$.
Then $\condD{1}{\ell}$ follows from Lemma \ref{Lem:direct_D} (i), (\ref{eq:nonzero_condition}), and $\condD{1}{\ell-1}$.
$\condD{2}{\ell}$ and $\condD{3}{\ell}$ are proved by calculating the values of the prepolarization by using Lemmas 
\ref{Lem:bilinear_induction} and \ref{Lem:direct_D} (i), and applying $\condD{2}{\ell-1}$ and $\condD{3}{\ell-1}$ respectively.

Let us prove $\condD{4}{\ell}$. 
Let $(b,c,d)\in \Z_{\ge 0}^{\times 3}$. We may assume $2\leq 2b\leq 3c\leq 2d \leq 6\ell$.
Set $B = 1+qA$.
By Lemmas \ref{Lem:bilinear_induction} and \ref{Lem:direct_D} (ii), we have
\[ ||e_2e^{(b,c,d)}v_\ell||^2 \in X+Y+2Z+W,
\]
where 
\begin{align*}
 &X= q^{-3c+3d-3}(q^{-3\ell}+q^{3\ell})(e^{(b-3,c-2,d-3)}v,e_2e^{(b,c,d)}v)B,\phantom{q^{-6b+12c}||e^{(b,c,d-3)}v||^2B+BBB}
\end{align*}
\begin{align*}
 &Y= ||e^{(b-3,c-2,d-3)}v||^2B+q^{-4b+6c+2d-4}||e^{(b,c,d-2)}v||^2B + q^{-6b+12c}||e^{(b,c,d-3)}v||^2B\phantom{BBB\,}\\
    &+q^{6c-12}||e^{(b,c-1,d-3)}v||^2B + q^{-2b+6c-4}||e^{(b-1,c-1,d-3)}v||^2B,
\end{align*}
\begin{align*}
 &Z= q^{-b+3c+2d-7}(e^{(b,c,d-2)}v,\!e_2e^{(b,c-1,d-2)}v)B+q^{-3b+9c-3}(e^{(b,c,d-3)}v,\!e_2e^{(b,c-1,d-3)}v)B\\
    &+q^{3b+3c-12}(e^{(b,c-1,d-3)}v,\!e_2e^{(b,c-2,d-3)}v)B+q^{b+3c-7}(e^{(b-1,c-1,d-3)}v,\!e_2e^{(b-1,c-2,d-3)}v)B,
\end{align*}
\begin{align*}
 &W= q^{-6c+6d-6}||e_2e^{(b,c,d)}v||^2B+q^{-2b+4d-4}||e_2e^{(b,c,d-1)}v||^2B\phantom{+q^{2d-4}||e_2e^{(b-1,c-1,d-2)}v||^2\,}\\
  &+q^{-4b+6c+2d+2}||e_2e^{(b,c,d-2)}v||^2B+ q^{2b+2d-10}||e_2e^{(b,c-1,d-2)}v||^2B\\
  &+q^{2d-4}||e_2e^{(b-1,c-1,d-2)}v||^2B+q^{-6b+12c+12}||e_2e^{(b,c,d-3)}v||^2B\\
  &+q^{6c-6}||e_2e^{(b,c-1,d-3)}v||^2B+q^{-2b+6c+2}||e_2e^{(b-1,c-1,d-3)}v||^2B\\
  &+ q^{6b-12}||e_2e^{(b,c-2,d-3)}v||^2B+q^{4b-10}||e_2e^{(b-1,c-2,d-3)}v||^2B\\
  &+ q^{2b-4}||e_2e^{(b-2,c-2,d-3)}v||^2B+q^6||e_2e^{(b-3,c-2,d-3)}v||^2B.
\end{align*}

Note that $X=0$ if $b<3$, and $X \subseteq q^{-6c+6d-6\ell}A$ if $b\geq 3$ by $\condD{5}{\ell-1}$.
$Y \subseteq A$ holds by $\condD{1}{\ell-1}$ and $\condD{2}{\ell-1}$, and $Z \subseteq A$ holds by $\condD{6}{\ell-1}$.
Finally $W \subseteq q^{\min\{0,6b-6c\}}$ for $b<3$ and $W\subseteq q^{\min\{0,6b-6c,-6c+6d-6\ell\}}$ for $b \geq 3$ 
are proved from $\condD{1}{\ell-1}$ and 
$\condD{6}{\ell-1}$. Hence $\condD{4}{\ell}$ follows.
The proofs of $\condD{5}{\ell}$ and $\condD{6}{\ell}$ are similar, and we omit the details.

Now Proposition \ref{Prop:hwv_G} is deduced from $\condD{1}{\ell}$--$\condD{4}{\ell}$ by exactly the same arguments
as in the case of type $G_2^{(1)}$.

\newcommand{\etalchar}[1]{$^{#1}$}
\def\cprime{$'$} \def\cprime{$'$} \def\cprime{$'$} \def\cprime{$'$}


\end{document}